\documentclass{article}
\usepackage{amsmath, amsthm, amssymb, latexsym,epsfig,amsthm,enumerate,multicol,wasysym}
\usepackage{graphicx} % Required for inserting images
\usepackage{xcolor}
\usepackage{hyperref}

\newcommand{\R}{\mathbb{R}}
\newcommand{\Z}{\mathbb{Z}}

\newcommand{\X}{\mathbb{X}}

\newcommand{\cC}{\mathcal{C}}
\newcommand{\cM}{\mathcal{M}}
\newcommand{\cQ}{\mathcal{Q}}

\newcommand{\diam}{\mathrm{diam}}

\newcommand{\cz}{\text{CZ}}
\newcommand{\touch}{\leftrightarrow}
\newcommand{\supp}{\text{supp}}
\newcommand{\dist}{\text{dist}}

\newtheorem{thm}{Theorem}

\newtheorem{lem}{Lemma}
\newtheorem{cor}{Corollary}

\title{An example related to Whitney's extension problem for $L^{2,p}(\R^2)$ when $1<p<2$}
\date{December 2023}
\author{Jacob Carruth and Arie Israel}

\begin{document}

\maketitle

\begin{abstract}
    In this paper, we prove the existence of a bounded linear extension operator $T: L^{2,p}(E) \rightarrow L^{2,p}(\mathbb{R}^2)$ when $1<p<2$, where $E \subset \mathbb{R}^2$ is a certain discrete set with fractal structure. Our proof makes use of a theorem of Fefferman-Klartag \cite{fefferman2023linear} on the existence of linear extension operators for radially symmetric binary trees.
\end{abstract}

\section{Introduction}

Let $\X(\R^n)$ be a space of continuous, real-valued functions on $\R^n$ equipped with a norm (or seminorm) $||\cdot||$. For any subset $\Omega \subset \R^n$, define the trace norm (or seminorm) of a continuous function $f: \Omega\rightarrow \R$ by
\[
||f||_{\X(\Omega)} = \inf\{ ||F|| : F|_\Omega = f\}.
\]
The \emph{trace space} $\X(\Omega)$ is then defined to be the space of continuous functions $f:\Omega \rightarrow \R$ with finite trace norm.

We say that a linear map $T : \X(\Omega) \rightarrow \X(\R^n)$ is a \emph{linear extension operator for $\X(\Omega)$} provided that $(Tf)|_\Omega = f$ for all $f \in \X(\Omega)$. We say that \emph{$\X(\Omega)$ admits a bounded linear extension operator} if there exists a constant $C$ (determined only by the function space $\X(\R^n)$) such that there exists a bounded linear extension operator for $\X(\Omega)$ with operator norm at most $C$. One formulation of the \emph{Whitney extension problem} for $\X(\R^n)$ asks whether  $\X(\Omega)$ admits a bounded linear extension operator for every subset $\Omega \subset \R^n$.

Let $\X(\R^n) = L^{m,p}(\R^n)$ denote the (homogeneous) Sobolev space of all functions $F : \R^n \rightarrow \R$ whose (distributional) partial derivatives of order $m$ belong to $L^p(\R^n)$, equipped with the seminorm $\|F \|_{L^{m,p}(\R^n)} = \| \nabla^m F \|_{L^p(\R^n)}$. We assume that $\frac{n}{m} < p < \infty$, so that functions $F \in L^{m,p}(\R^n)$ are continuous -- this ensures that the trace space $L^{m,p}(\Omega)$ is well-defined for an arbitrary subset $\Omega \subset \R^n$.

In \cite{fefferman2014sobolev}, the second-named author, C. Fefferman, and G.K. Luli proved that $L^{m,p}(\Omega)$ admits a bounded linear extension operator when $p>n$, with operator norm at most a constant $C = C(m,n,p)$, for an arbitrary subset $\Omega \subset \R^n$. When $p$ is in the range $\frac{m}{n} < p \le n$, however, it is unknown whether such an operator exists in general.
%\footnote{Under suitable geometric conditions on $\Omega$, the existence of a bounded linear extension operator is known for a range of $p$ related to the dimension of $\Omega$. For instance, see \cite{stein1970singular} for the case when $\Omega$ is a Lipschitz domain, or \cite{} for the case when $\Omega$ is a smooth submanifold of $\R^n$.}

The first nontrivial case in this range is the space $\X = L^{2,p}(\R^2)$ for $1 < p < 2$ (when $p=2$, it is easy to construct a bounded linear extension operator thanks to the Hilbert space structure). In this paper, we construct a bounded linear extension operator for $L^{2,p}(E)$ for a particular set $E \subset \R^2$ to be defined below. This set $E$ appears to be exceptional. For a variety of sets $\Omega \subset \R^2$ it is possible to construct a bounded linear extension operator for $L^{2,p}(\Omega)$ by making certain analogies with the case $p>2$. Our construction for $L^{2,p}(E)$, however, requires completely new ideas.

We now define the set $E$. We introduce a real number $\varepsilon \in (0,1/2)$. We assume that $\varepsilon$ is smaller than some absolute constant. For an integer $L \geq 1$, we define
\[
\Delta = \varepsilon^L.
\]
The set $E$ is of the form
\[
E = E_1 \cup E_2,
\]
where
\begin{flalign*}
    &E_1 = \{(\Delta \Z) \cap [-1,1]\}\times\{0\} \; \mbox{and}\\
    &E_2 = \Big\{(z, \Delta) : z = \sum_{\ell=1}^L s_\ell \varepsilon^\ell, s_\ell \in \{-1,+1\}\Big\},
\end{flalign*}
Accordingly, $E_1$ is a set of equispaced $\Delta$-separated points on the line $x_2 =0$, and $E_2$ is a set of points of separation $\geq \Delta$ on the line $x_2 = \Delta$. For convenience, we assume that for every $(x, \Delta) \in E_2$ we have $(x,0) \in E_1$. We can ensure this by, e.g., taking $\varepsilon = 1/N$ for a large, positive integer $N$.

In this paper, we prove the following theorem.
\begin{thm}\label{thm: main}
Let $1 < p < 2$. There exists a bounded linear extension operator $T:L^{2,p}(E) \rightarrow L^{2,p}(\R^2)$, depending on $p$, satisfying
\[
    \|Tf\|_{L^{2,p}(\R^2)} \le C \|f\|_{L^{2,p}(E)}\;\text{for any}\; f \in L^{2,p}(E)
    \]
    for some constant $C$ depending only on $p$.
\end{thm}

We now say a bit about the construction of the operator $T$. We fix a square $Q^0\subset \R^2$ containing $E$ with diameter $\sim 1$. Given $f: E\rightarrow \R$, we will describe how to produce a function $F \in L^{2,p}(Q^0)$ satisfying $F|_E = f$ and having the property that
\begin{equation}\label{eq: main 1}
||F||_{L^{2,p}(Q^0)}\le C ||G||_{L^{2,p}(\R^2)}
\end{equation}
for any $G \in L^{2,p}(\R^2)$ for which $G|_E = f$. Once we accomplish this, it is not difficult to deduce Theorem \ref{thm: main}.

We first partition the square $Q^\circ$ into a Calder\'{o}n-Zygmund decomposition $\cz$, which is a finite family of dyadic squares $Q \subset Q^0$ with pairwise disjoint interiors. Our decomposition is a version of the Whitney decomposition of $Q^0 \setminus E$ truncated at scale $\sim \Delta$. Briefly, every $Q \in \cz$ has sidelength $\delta_Q \approx \Delta + \dist(Q,E)$ and satisfies $\#(1.1 Q \cap E) \leq 1$. In particular, the smallest cubes in $\cz$ have sidelength proportional to $\Delta$ and distance to $E$ bounded by $C \Delta$ -- see section \ref{sec:cz} for further details. Our function $F \in L^{2,p}(Q^0)$ takes the form
\begin{equation}\label{F:defn}
F(x) = \sum_{Q \in \cz} \theta_Q(x)\cdot P_Q(x),
\end{equation}
where $\{\theta_Q\}_{Q \in \cz}$ is a Whitney partition of unity subordinate to $\cz$ -- more precisely,  (1) $\sum_Q \theta_Q = 1$ on $Q^0$, (2) $\supp(\theta_Q) \subset 1.1 Q$, and (3) $|\nabla^k \theta_Q| \lesssim \delta_Q^{-k}$ for each $Q \in \cz$ -- and each $P_Q$ is an affine polynomial satisfying 
\begin{equation}
    \label{loc_ext:cond}
    P_Q  = f \mbox{ on } 1.1Q \cap E.
\end{equation} 
Using \eqref{loc_ext:cond} and the support condition on $\theta_Q$, we see that the function $F$ will satisfy $F|_E = f$. For each $Q \in \cz$, $P_Q$ is of the form 
\[
P_Q(z) = L_Q(z) + \eta_Q\cdot z^{(2)}\;\text{for}\; z = (z^{(1)}, z^{(2)}) \in \R^2,
\]
where $\eta_Q$ is a real number and $L_Q$ is an affine polynomial satisfying $\partial_2 L_Q = 0$. Due to the structure of the set $E_1$, our hand is essentially forced when it comes to choosing the $L_Q$'s -- specifically, we  demand that each $L_Q$ agrees with the function $f$ at two suitably chosen points of the set $E_1$. The bulk of the work, then, is in choosing the numbers $\{\eta_Q\}_{Q \in \cz}$ to ensure a certain natural \emph{compatibility condition} on the family $\{P_Q\}_{Q \in \cz}$, namely, the  condition:
\begin{equation}\label{comp:cond} \sum_{Q,Q' \in \cz, Q \leftrightarrow Q'} \delta_Q^{2-p} | \nabla P_Q - \nabla P_{Q'}|^p \lesssim  \| G \|_{L^{2,p}(\R^2)}^p,
\end{equation}
where $Q \leftrightarrow Q'$ denotes the property that $1.1Q \cap 1.1Q' \neq \emptyset$ for two squares $Q,Q' \in \cz$ (we say that $Q$ ``touches'' $Q'$), and $G \in L^{2,p}(\R^2)$ is any function satisfying $G|_E = f$. Estimate \eqref{comp:cond} is the essential ingredient in the proof of the norm bound \eqref{eq: main 1} for the function $F$ defined in \eqref{F:defn}.

Roughly speaking, each $\eta_Q$ corresponds to the average $x^{(2)}$-partial derivative of $F$ on the square $Q$. Since the value of $f$ near points of the set $E_2$ is our only source of information about the $x^{(2)}$-partial derivative of $F$, we make use of a natural hierarchical clustering $\cC$ of the set $E_2$ to choose the $\eta_Q$'s. 

Elements $C \in \cC$ correspond to the subsets of $E_2$ of the form $C(s_1,\cdots,s_k) := \{ (z,\Delta) : z = \sum_{\ell=1}^L s_\ell \epsilon^\ell , \; s_{k+1},\cdots,s_L \in \{ -1, +1\} \}$ for fixed $(s_1,\cdots,s_k) \in \{-1,+1\}^k$. The smallest clusters correspond to the $2^L$ different singleton sets in $E_2$ (when $k=L$), and the largest cluster corresponds to $C=E_2$ (when $k=0$). Distinct clusters are ``well-separated'' in the sense that $\dist(C_1,C_2) \gg \max\{ \diam(C_1),\diam(C_2)\}$ for $C_1,C_2 \in \cC$, $C_1 \neq C_2$.

The set of clusters $\cC$ has the structure of a rooted binary tree of depth $L$. The clusters corresponding to the leaves of the tree are the singletons $\{z\}$ ($z \in E_2$), and $E_2$ is the root of $\cC$.

To every square $Q \in \cz$, we associate a cluster $C_Q \in \cC$. If $1.1Q \cap E_2 \neq \emptyset$ then $Q$ is associated to the singleton cluster $C_Q = \{x_Q\}$ where $x_Q$ is a point of $1.1 Q \cap E_2$. To every other square $Q$ we associate a cluster $C_Q \in \cC$ that is nearby to $Q$.

We then introduce real numbers $\{\eta_C\}_{C \in \cC}$ and postulate that
\begin{equation}\label{etaQ:cond}
\eta_Q = \eta_{C_Q}\;\text{for every}\; Q \in \cz.
\end{equation}

For each $z = (z^{(1)}, \Delta) \in E_2$, we define
\begin{equation*}
\eta_{\{z\}} = \frac{f(z) - f(z^{(1)}, 0)}{\Delta}.
\end{equation*}
This is natural, as for squares $Q \in \cz$ satisfying $1.1Q \cap E_2 \neq \emptyset$ we will need to choose $\eta_Q = \eta_{\{x_Q\}}$ where $x_Q \in 1.1Q \cap E_2$, to ensure that $P_Q$ satisfies \eqref{loc_ext:cond}.

Now that we have chosen $(\eta_{\{z\}})_{z \in E_2}$, we need to choose the remaining $\eta_C$ so as to establish the compatibility condition \eqref{comp:cond} for the $P_Q$. We use the following theorem of C. Fefferman and B. Klartag, proved in \cite{fefferman2023linear}.

\begin{thm}[Fefferman-Klartag, 2023]\label{thm: fk}
    Consider a full, binary tree of depth $N$ with vertices $V$. Let $V_k$ denote the set of vertices of depth $k$. For any $v \in V_k$ ($k=1,\dots,N$), let $\pi(v)$ denote the parent of $v$.

    Write $\partial V \subset V$ to denote the set of leaves of $V$ (i.e., $\partial V = V_N$). Let $\R^V$ and $\R^{\partial V}$ denote the sets of real-valued functions on $V$ and $\partial V$, respectively.

    Given weights $w_1, \dots, w_N >0$, define the $L^{1,p}(V)$-seminorm by
    \begin{equation}\label{eq: smnrm}
    \|\Phi\|_{L^{1,p}(V)} = \bigg( \sum_{k=1}^N w_k \sum_{v \in V_k}|\Phi(\pi(v)) - \Phi(v)|^p\bigg)^{1/p}\;\text{for any}\; \Phi \in \R^V.
    \end{equation}
     Define the $L^{1,p}$ trace seminorm by
    \[
    \|\phi\|_{L^{1,p}(\partial V)} = \inf \{ \|\Phi\|_{L^{1,p}(V)} : \Phi|_E = \phi\}\;\text{for any}\; \phi \in \R^{\partial V}.
    \]

    There exists a linear operator $H: \R^{\partial V} \rightarrow \R^V$ with the following properties:
    \begin{enumerate}
        \item $H$ is an extension operator, i.e.,
        \[
        H\phi|_{\partial V} = \phi\;\text{for any}\; \phi \in \R^{\partial V}.
        \]
        \item For any $\phi \in \R^{\partial V}$, we have
        \[
        \|H\phi\|_{L^{1,p}(V)} \lesssim \|\phi\|_{L^{1,p}(\partial V)}.
        \]
    \end{enumerate}
\end{thm}

We apply Theorem \ref{thm: fk} to the binary tree $\cC$ equipped with a natural set of edge weights, and choose the $\eta_C$ by applying the extension operator $H$ to the data $\{\eta_{\{z\}}\}_{z \in E_2}$; because the resulting choice of $\{\eta_C\}$  minimizes the seminorm \eqref{eq: smnrm}, one can show the resulting polynomials $\{P_Q\}$ defined in terms of the $\{\eta_Q\}$ (which are related to the $\{\eta_C\}$ in \eqref{etaQ:cond}) will satisfy \eqref{comp:cond}. As mentioned before, the norm bound \eqref{eq: main 1} follows as a consequence of this condition.

Because the $\eta_{\{z\}}$ depend linearly on $f$, and the $\eta_C$ depend linearly on $\{ \eta_{\{z\}}\}_{z \in E_2}$, it is obvious that the $P_Q$  depend linearly on $f$, and so, $F$ depends linearly on $f$.

It is finally trivial to extend $F : Q^0 \rightarrow \R$ to a function $\widetilde{F} : \R^2 \rightarrow \R$ satisfying $\| \widetilde{F} \|_{L^{2,p}(\R^2)} \leq C \| F \|_{L^{2,p}(Q^0)}$. Since $F \in L^{2,p}(Q^0)$ is an extension of $f$ satisfying the norm bound \eqref{eq: main 1}, we deduce that $\widetilde{F} \in L^{2,p}(\R^2)$ is an extension of $f$ satisfying $\| \widetilde{F} \|_{L^{2,p}(\R^2)} \leq C \| G \|_{L^{2,p}(\R^2)}$ for any $G \in L^{2,p}(\R^2)$ with $G|_E = f$. Thus, we cam set $Tf = \widetilde{F}$, and $T$ is a bounded linear extension operator for $L^{2,p}(E)$.

This concludes our overview of the proof of Theorem \ref{thm: main}.

We remark that when $\varepsilon$ is bounded away from a certain ``critical value'' $\varepsilon_0 = (1/2)^{1/(2-p)}$, it is possible to prove Theorem \ref{thm: main} without invoking Theorem \ref{thm: fk}. In particular, one can construct an extension operator for $L^{2,p}(E)$ by making certain analogies with the case $p>2$. When $\varepsilon = \varepsilon_0$, however, we know of no such construction.

Throughout this paper we write $C, C', C'', \dots$ to denote positive constants. Such constants are allowed to depend only on $p$ (but not on $\varepsilon, L$). We write $C_X, C'_X, \dots$ for positive constants depending on some quantity $X$. For positive real numbers $A, B$ we write $A \lesssim B$ if there exists a constant $C$ such that $A \le C B$ and $A \lesssim_X B$ if $A \le C_X B$. We write $A \approx B$ if $A \lesssim B$ and $A\lesssim B$.

We thank Charles Fefferman, Anna Skorobogotova, and Ignacio Uriarte-Tuero, for helpful conversations. We also thank Pavel Shvartsman for suggesting that we use the Hardy-Littlewood maximal function; this greatly simplified some of the arguments.

\section{Preliminaries}\label{sec: prelim}

For a (Lebesgue) measurable function $F$ defined on a measurable set $S \subset \R^2$ with $|S| > 0$, we write $(F)_S := |S|^{-1} \int_S F \ dx$. We begin by stating a couple of standard inequalities for Sobolev spaces.  For details see, e.g., \cite{gilbarg1977elliptic}.

Given an annulus $A = \{ x \in \R^2 : r \leq |x-x_0| \leq R \}$ with inner radius $r$ and outer radius $R$,  the \emph{thickness ratio} of $A$ is defined to be the quantity $R/r$.

\begin{lem}\label{lem: preponc}
    Let $\Omega \subset \R^2$ be a square, a ball, or an annulus with thickness ratio at most $C_0 \in [1,\infty)$. Then the following hold.
    \begin{enumerate}
        \item Let $1 < r < 2$, and $r' = 2r/(2-r)$. Then for any $F \in L^{2,r}(\R^2)$ we have
        \[
        \bigg( \int_\Omega | \nabla F(x) - (\nabla F)_\Omega |^{r'}\ dx \bigg)^{1/r'} \lesssim_{r,C_0} \|F\|_{L^{2,r}(\Omega)}.
        \]
        \item Let $q > 2$. Then for any $F\in L^{1,q}(\Omega)$ we have (after potentially redefining $F$ on a set of measure 0) that
        \[
        |F(x) - F(y)| \lesssim_{q,C_0} |x-y|^{1-2/q} \|F\|_{L^{1,q}}\;\text{for any}\; x,y \in \Omega.
        \]
    \end{enumerate}
\end{lem}
We use these inequalities to prove the following basic lemma.
\begin{lem}\label{lem: poincare}
    Let $\Omega\subset \R^2$ be a square, a ball, or an annulus with thickness ratio at most $C_0 \in [1,\infty)$ and let $1 < r< 2$. For any $F \in L^{2,r}(\Omega)$ and any $y \in \Omega$, we define a function $T_{y,\Omega}(F):\R^2 \rightarrow \R$ by
    \[
    T_{y,\Omega}(F)(x) = F(y) + (\nabla F)_{\Omega}\cdot (x-y).
    \]
    We then have
    \[
    |T_{x,\Omega}(F)(y) - T_{z, \Omega}(F)(y)|
        \lesssim_{r,C_0}  \|F\|_{L^{2,r}(\Omega)} |x-z|^{2-2/r}\;\text{for any}\; x,z \in \Omega.
    \]
    In particular, 
    \[
    \| F - T_{y,\Omega}(F)\|_{L^\infty(\Omega)} \lesssim_{r,C_0} \diam(\Omega)^{2-2/r}\|F\|_{L^{2,r}(\Omega)}.
    \]
\end{lem}
\begin{proof}
    Let $r' = 2r/(2-r)$. For any $x \in \R^2$ define $G_x : \R^2 \rightarrow \R$ by
    \[
    G_x(y) = T_{y,\Omega}(F)(x).
    \]
    By Part 1 of Lemma \ref{lem: preponc}, we have
    \[
    \|G_x\|_{L^{1,r'}(\Omega)} \lesssim_{r,C_0} \|F\|_{L^{2,r}(\Omega)}.
    \]
    Note that $r'>2$, since $r < 2$. Applying Part 2 of Lemma \ref{lem: preponc} to the function $G_x$ proves the lemma.
\end{proof}

Let $B(z,r)$ denote the ball of radius $r>0$ centered at $z \in \R^2$, and let $\cM$ denote the uncentered Hardy-Littlewood maximal operator, i.e.,
\[
(\cM f)(x) = \sup_{B(z,r) \ni x} \frac{1}{|B(z,r)|} \int_{B(z,r)} f(y)\ dy\;\text{for any}\; f \in L^1_{\text{loc}}(\R^2).
\]
Recall that $\cM$ is a bounded operator from $L^q(\R^2)$ to $L^q(\R^2)$ for any $1 < q\le  \infty$ (see, e.g., \cite{stein1993harmonic}).

\section{The CZ decomposition}\label{sec:cz}

We will work with squares in $\R^2$; by this we mean an axis parallel square of the form $Q = [a_1, b_1) \times [a_2, b_2)$. We let $\delta_Q$ denote the sidelength of such a square $Q$. To \emph{bisect} a square $Q$ is to partition $Q$ into squares $Q_1, Q_2, Q_3 ,Q_4$, where $\delta_{Q_i} = \delta_Q / 2$ for each $i = 1,2,3,4$. We refer to the $Q_i$ as the \emph{children} of $Q$. 

We define a square $Q^0 = [-4, 4) \times [-4,4)$; note that $E \subset Q^0$. A \emph{dyadic} square $Q$ is one that arises from repeated bisection of $Q^0$. Every dyadic square $Q \neq Q^0$ is the child of some square $Q'$; we call $Q'$ the \emph{parent} of $Q$ and denote this by $(Q)^+ = Q'$.

We say that two dyadic square $Q, Q'$ \emph{touch} if $(1.1Q \cap 1.1Q') \ne \emptyset$. We write $Q \leftrightarrow Q'$ to denote that $Q$ touches $Q'$.  

For any dyadic square $Q$, we define a collection $\cz(Q)$, called the \emph{Calder\'{o}n-Zygmund decomposition of} $Q$, by setting
\[
\cz(Q) = \{Q\} \;\text{if } \#(3Q\cap E) \le 1,
\]
and
\[
\cz(Q) = \bigcup \{\cz(Q') : (Q')^+ = Q\}\;\text{if}\; \#(3Q \cap E) \ge 2.
\]

We write $\cz = \cz(Q^0)$. Note that $\cz \neq \{ Q^0\}$ because $\#(3Q^0 \cap E) = \# E \geq 2$. Then $\cz$ is a partition of $Q^0$ into dyadic squares $Q$ satisfying
\begin{equation}\label{eq: cz 1}
\delta_Q \ge \frac{\Delta}{9}.
\end{equation}
To see this, observe that for any $Q \in \cz$ we have $\# (3Q^+ \cap E) \ge 2$ and $3Q^+ \subset 9Q$. Since points of $E$ are $\Delta$-separated, we have $\delta_Q \ge \Delta /9$, as claimed.

We now summarize some basic properties of the collection $\cz$.

\begin{lem}\label{lem: cz}
    The collection $\emph{CZ}$ has the following properties:
    \begin{enumerate}
    \item For any $Q \in \emph{CZ}$, we have $\#(1.1Q \cap E) \le 1$ and $\#(3Q^+ \cap E ) \ge 2$.
        \item For any $Q, Q' \in \emph{CZ}$ with $Q \touch Q'$, we have $\frac{1}{2}\delta_Q \le \delta_{Q'}\le 2 \delta_Q$.
        \item For any $Q \in \emph{CZ}$, we have
        \[
        \#\{Q' : Q\touch Q'\} \lesssim 1.
        \]
        \item For any $x \in \R^2$,
        \[
        \#\{Q \in \emph{CZ} : x \in 1.1Q\} \lesssim 1.
        \]
        \item For any $Q \in \emph{CZ}$ with $\#(1.1Q \cap E) = 0$, we have $\delta_Q \approx \emph{dist}(Q,E)$.
    \end{enumerate}
\end{lem}
We omit the proof of Lemma \ref{lem: cz}, as this type of decomposition is standard in the literature; see, e.g., \cite{fefferman2020fitting}.

We say that $Q \in \cz$ is a \emph{boundary square} if $1.1 Q \cap \partial Q^0 \ne \emptyset$. Denote the set of boundary squares by $\partial \cz$. 

Observe that every $Q \in \partial \cz$ satisfies $\delta_Q \ge 1$. Indeed, this follows because $E \subset [-1,1] \times [-1,1]$, and if $Q$ is a dyadic square intersecting the boundary of $Q^0$ with $\delta_Q = 1$ then $3Q$ is disjoint from $E$. We denote $z_0 := (-1,0)$, $w_0 := (1,0)$ to be the points of maximal separation in $E_1$. Note that 
\begin{equation}\label{z0w0:prop}
z_0, w_0 \in 50 Q \mbox{ for all } Q \in \partial \cz.
\end{equation}

We say that a square $Q \in \cz$ is of \emph{Type I} if $\#(1.1Q \cap E_1) = 1$, \emph{Type II} if $\#(1.1Q \cap E_2) = 1$, and \emph{Type III} if $\#(1.1Q \cap E) = 0$. We denote the collection of squares of Type I, II, and III by $\cz_I$, $\cz_{II}$, and $\cz_{III}$, respectively. These collections form a partition of $\cz$ because $\#(3Q \cap E) \leq 1$ for any $Q \in \cz$ by construction.

Observe that $\partial \cz \subset \cz_{III}$.

From Property 1 of Lemma \ref{lem: cz}, it is easy to deduce that
\begin{equation}\label{eq: cz 2}
\delta_Q \approx \Delta \;\text{for every}\; Q \in \cz_I \cup \cz_{II}.
\end{equation}
From Property 5 of Lemma \ref{lem: cz} we have
\[
\delta_Q \approx \dist(Q,E)\;\text{for every}\; Q \in \cz_{III}.
\]
Combining this with \eqref{eq: cz 1}, and using that points of $E$ are $\Delta$-separated, we get
\[
\delta_Q \approx \dist(Q,E_1)\;\text{for every}\;Q \in \cz_{III}.
\]
Combining this with \eqref{eq: cz 1}, \eqref{eq: cz 2} gives
\begin{equation}\label{eq: cz 3}
    \delta_Q \approx (\Delta + \dist(Q, E_1))\;\text{for any} \; Q \in \cz.
\end{equation}

For each $Q \in \cz_{II}$ we let $x_Q$ be the unique point in $(1.1Q) \cap E = (1.1Q) \cap E_2$. Note that $x_Q$ is undefined for $Q \in \cz \setminus \cz_{II}$.

For each $Q \in \cz$ we associate a pair of points $z_Q,w_Q \in E_1$. We list the key properties of these points in the next lemma.

\begin{lem} \label{lem: bpoint}
For each $Q \in \emph{CZ}$ there exist points $z_Q, w_Q \in (50Q) \cap E_1$, satisfying the conditions below.
\begin{enumerate}
    \item $|z_Q - w_Q| \approx \delta_Q$.
    \item If $Q \in \cz_I$ then $z_Q \in (1.1Q) \cap E_1$.
    \item If $Q \in \cz_{II}$ then $z_Q = (x_Q^{(1)}, 0)$, where $x_Q = (x_Q^{(1)}, \Delta)$ is the unique point in $ 1.1Q \cap E_2$.
    \item If $Q \in \partial \cz$ then $z_Q = z_0 $ and $w_Q = w_0$.
\end{enumerate}

\end{lem}

\begin{proof}
We first check that for each $Q \in \cz$ there exist points $z_Q,w_Q \in (50Q) \cap E_1$ satisfying $|z_Q - w_Q| \approx \delta_Q$. Fix $Q \in \cz$. By definition of the Calder\'{o}n-Zygmund stopping rule, the set $3Q^+ \cap E$ is nonempty; also, $3Q^+ \subset 9Q$ by properties of dyadic squares, so $9Q \cap E$ is nonempty. Thus, either $9 Q \cap E_1 \neq \emptyset$ or $9Q \cap E_2 \neq \emptyset$. We claim that $30 Q \cap E_1$ is nonempty in either case. This is true if $9Q \cap E_1 \neq \emptyset$. On the other hand if $9Q \cap E_2 \neq \emptyset$, we let $y_Q$ be a point of $9 Q \cap E_2$, and note that $\dist(y_Q,E_1) = \Delta \leq 9\delta_Q$, hence, $(30Q) \cap E_1 \neq \emptyset$. In any case, we have shown $(30Q) \cap E_1 \neq \emptyset$. Once that is established, by the definition of $E_1$ as a sequence of $\Delta$-separated points and because $9\delta_Q \geq \Delta$, we see that $50Q \cap E_1$ contains at least two points. Thus, we can order the points of $(50Q) \cap E_1 \subset \R \times \{0\}$ according to the order on the real number line and let $z_Q$ and $w_Q$ be the distinct minimal and maximal points in the set $(50Q) \cap E_1$. Obviously then $|z_Q - w_Q | \approx \delta_Q$.

We make small modifications to the construction to establish conditions 2-4 of the lemma.

If $Q \in \cz_I$ then we instead select $z_Q \in (1.1)Q \cap E_1$ and then let $w_Q \in E_1$ be adjacent to $z_Q$ so that $|z_Q - w_Q| = \Delta \approx \delta_Q$. Since $9\delta_Q \geq \Delta$, it follows that $w_Q \in 50 Q$, so $w_Q \in 50Q \cap E_1$, as desired.

If $Q \in \cz_{II}$ then we instead choose $z_Q \in E_1$ to be the point directly below $x_Q \in (1.1Q) \cap E_2$ so that $|z_Q - x_Q| = \Delta$. We then let $w_Q \in E_1 $ be a point adjacent to $z_Q$, so that $|w_Q - z_Q| = \Delta \approx \delta_Q $. As before, since $x_Q \in 1.1Q$ and $9\delta_Q \geq \Delta$, we see that $z_Q,w_Q \in 50Q$, as desired.

If $Q \in \partial \cz$ then we define $z_Q = z_0$ and $w_Q = w_0$, where $z_0 = (-1,0)$ and $w_0 = (1,0)$. Note that $|z_Q - w_Q| \approx 1 \approx \delta_Q$. Before we established that $z_0, w_0 \in 50 Q$, so, in particular, $z_Q,w_Q \in 50 Q \cap E_1$, as desired.

\end{proof}

\section{Constructing the interpolant}
Let $f: E \rightarrow \R$. In this section, we will construct a function $F: Q^0 \rightarrow \R$ satisfying $F|_E = f$.

Let $\{\theta_Q\}_{Q\in \cz}$ be a partition of unity subordinate to $\cz$ constructed so that the following properties hold. For any $Q \in \cz$, we have:
\begin{enumerate}
    \item[\textbf{(POU1)}] $\supp(\theta_Q) \subset 1.1Q$.
    \item[\textbf{(POU2)}] For any $|\alpha| \le 2$, $\|\partial^\alpha \theta_Q\|_{L^\infty}\lesssim \delta_Q^{-|\alpha|}.$
    \item[\textbf{(POU3)}]$ 0 \le \theta_Q \le 1$.
\end{enumerate}
For any $x \in Q^0$, we have
\begin{enumerate}
\item[\textbf{(POU4)}]$\sum_{Q \in \cz} \theta_Q(x) =1.$
\end{enumerate}

Our interpolant $F$ will then be of the form
\begin{equation}\label{eq: interp 1}
    F = \sum_{Q \in \cz} P_Q \theta_Q,
\end{equation}
where each $P_Q$ is an affine polynomial on $\R^2$ of the form
\[
P_Q(x) = L_Q(x) + \eta_Q\cdot x^{(2)}\;\text{for}\; x = (x^{(1)}, x^{(2)}) \in \R^2.
\]
Here, $L_Q$ is an affine polynomial satisfying $\partial_2 L_Q = 0$ and $\eta_Q$ is a real number.

In the previous section we have associated to each square $Q \in \cz$ the points $z_Q, w_Q \in E_1$. We now define $L_Q$ to be the unique affine polynomial satisfying:
\begin{equation}\label{LQ:defn}
    L_Q|_{\{z_Q, w_Q\}} = f|_{\{z_Q, w_Q\}} \mbox{  and  } \partial_2 L_Q = 0.
\end{equation}
Thanks to Property 4 of Lemma \ref{lem: bpoint} there exists an affine polynomial $L_0$ for which
\[
L_Q = L_0 \;\text{for all}\; Q \in \partial \cz.
\]
We will choose the $\eta_Q$ later in Sections \ref{sec: eta} and \ref{sec: eta tree}.

We now compute a simple upper bound for the $L^{2,p}$-seminorm of $F$.

\begin{lem}[The Patching Lemma]\label{lem: patching}
    The function $F$ satisfies
    \[
    \|F\|^p_{L^{2,p}(Q^0)} \lesssim \sum_{\substack{Q, Q' \in \cz:\\ Q\touch Q'}} \Big\{ \|L_Q - L_{Q'}\|_{L^\infty(Q)}^p \delta_Q^{2-2p} + |\eta_Q - \eta_{Q'}|^p \delta_Q^{2-p}\Big\}. 
    \]
\end{lem}

\begin{proof}
    Fix a square $Q' \in \cz$. Observe that 
    \[
    F(x) = \sum_{Q \in \cz} \theta_Q(x) [P_Q(x) - P_{Q'}(x)] + P_{Q'}(x).
    \]
    By Property 4 of Lemma \ref{lem: cz}, there are a bounded number of squares $Q \in \cz$ for which $x \in (1.1Q) \cap Q'$. Therefore, by \textbf{(POU1)}, there are a bounded number of $Q \in \cz$ with $\supp(\theta_Q) \cap Q' \ne \emptyset$. Taking $2^{\text{nd}}$ derivatives, using \textbf{(POU2)}, and integrating $p^{\text{th}}$ powers then gives
    \[
    \|F\|_{L^{2,p}(Q')}^p \lesssim \sum_{\substack{Q\in \cz:\\ Q \touch Q'} } \big\{\delta_Q^{2 - 2p} \|P_Q - P_{Q'}\|^p_{L^\infty(Q)} + \delta_Q^{2-p}| \nabla (P_Q - P_{Q'})|\big\}.
    \]
    For any affine polynomial $P$ we have $|\nabla P| \le \delta_Q^{-1} || P||_{L^\infty(Q)}$, and thus
    \[
    \|F\|_{L^{2,p}(Q')}^p \lesssim \sum_{\substack{Q\in \cz:\\ Q \touch Q'} } \delta_Q^{2 - 2p} \|P_Q - P_{Q'}\|^p_{L^\infty(Q)}.
    \]
    By the triangle inequality, and the fact that $|z^{(2)}| \leq C \delta_Q$ whenever $z \in 1.1Q$,
\begin{align*}
    \| P_Q - P_{Q'} \|_{L^\infty(Q)} &\leq \| L_Q - L_{Q'} \|_{L^\infty(Q)} + |\eta_Q - \eta_{Q'}| \| z^{(2)}\|_{L^\infty(Q)} \\
    &\lesssim \| L_Q - L_{Q'} \|_{L^\infty(Q)} + |\eta_Q - \eta_{Q'} | \delta_Q.
\end{align*}
    Therefore,
\[
    \|F\|_{L^{2,p}(Q')}^p \lesssim \sum_{\substack{Q\in \cz:\\ Q \touch Q'} } \left\{\delta_Q^{2 - 2p} \|L_Q - L_{Q'}\|^p_{L^\infty(Q)} + \delta_Q^{2-p} |\eta_Q - \eta_{Q'}|^p \right\}.
\]
    Since $\cz$ is partition of $Q^0$, summing over $Q' \in \cz$ proves the lemma.
\end{proof}

\section{Establishing compatibility of the $L_Q$}\label{sec: lq}

In this section we will prove that for any function $G \in L^{2,p}(\R^2)$ with $G|_E = f$ we have
\begin{equation}\label{eq: lq1}
\sum_{\substack{Q, Q' \in \cz:\\Q\touch Q'}} \|L_Q - L_{Q'}\|_{L^\infty(Q)}^p \delta_Q^{2-2p} \lesssim \|G\|_{L^{2,p}(\R^2)}^p.
\end{equation}

Fix some $r$ satisfying $1 < r < p$. We claim that for any $Q, Q' \in \cz$ with $Q\touch Q'$ we have
\begin{equation}\label{eq: lq1.1}
\| L_Q - L_{Q'}\|_{L^\infty(Q)}^r \lesssim \delta_Q^{2r-2} \|G\|_{L^{2,r}(250Q)}^r.
\end{equation}

Observe that for any $x \in Q$ we have
\begin{equation*}
L_Q(x) - L_{Q'}(x) = f(z_Q) + \partial_1 L_Q\cdot (x-z_Q) - f(z_{Q'}) - \partial_1 L_{Q'}\cdot (x - z_{Q'}).
\end{equation*}
Therefore, since $G(z_{Q'}) = f(z_{Q'})$ and $\partial_2 L_Q = \partial_2 L_{Q'}=0$, we have
\begin{align*}
    \|L_Q - L_{Q'}\|_{L^\infty(Q)} \le & | f(z_Q) - T_{z_{Q'}, 250Q}(G)(z_Q)| \\ &+ |(\partial_1 G)_{250Q} - \partial_1 L_Q|\cdot\delta_Q + |(\partial_1 G)_{250Q} - \partial_1 L_{Q'}|\cdot \delta_Q.
\end{align*}
Recall that for any $Q, Q' \in \cz$ with $Q \leftrightarrow Q'$ we have $\frac{1}{2}\delta_{Q'} \le \delta_Q \le 2 \delta_{Q'}$ (see Lemma \ref{lem: cz}), and therefore $50Q' \subset 250 Q$. In particular, we have 
\[
\{z_Q, w_Q, z_{Q'}, w_{Q'}\} \subset (250 Q)\cap E_1.
\]
Since $G|_E = f$, Lemma \ref{lem: poincare} then implies
\begin{align*}
    &|f(z_Q) - T_{z_{Q'}, 250 Q}(G)(z_Q)| \lesssim \delta_Q^{2-2/r}\|G\|_{L^{2,r}(250 Q)},\\
    &|f(z_Q) - f(w_Q) - (\partial_1 G)_{250 Q} \cdot(z_Q - w_Q)| \lesssim \delta_Q^{2-2/r}\|G\|_{L^{2,r}(250 Q)},\\
    & |f(z_{Q'}) - f(w_{Q'}) - (\partial_1 G)_{250 Q}\cdot(z_{Q'} - w_{Q'})| \lesssim \delta_Q^{2-2/r}\|G\|_{L^{2,r}(250 Q)};
\end{align*}
we deduce \eqref{eq: lq1.1}.

The inequality \eqref{eq: lq1.1} implies that
\[
\| L_Q - L_{Q'}\|_{L^\infty(Q)}^r \lesssim \delta_Q^{2r} \cM(|\nabla^2 G|^r)(z)\;\text{for any}\; z \in Q,
\]
where $\cM$ is the Hardy-Littlewood maximal operator (see Section \ref{sec: prelim}). Taking $(p/r)$-th powers and integrating gives
\[
\| L_Q - L_{Q'}\|_{L^\infty(Q)}^p \lesssim \delta_Q^{2p-2} \|\cM(|\nabla^2 G|^r)\|_{L^{p/r}(Q)}^{p/r}.
\]
By Properties 2 and 3 of Lemma \ref{lem: cz}, we deduce
\[
\sum_{\substack{Q, Q' \in \cz:\\Q \touch Q'}} \|L_Q - L_{Q'}\|_{L^\infty(Q)}^p \delta_Q^{2-2p} \lesssim \sum_{Q \in \cz} \|\cM(|\nabla^2 G|^r)\|_{L^{p/r}(Q)}^{p/r}.
\]
Since $\cz$ is a pairwise disjoint collection of squares, we have
\[
\sum_{Q \in \cz} \|\cM(|\nabla^2 G|^r)\|_{L^{p/r}(Q)}^{p/r} \le \| \cM(|\nabla^2 G|^r)\|_{L^{p/r}(\R^2)}^{p/r}.
\]
We use the boundedness of the Hardy-Littlewood maximal operator from $L^{q}$ to $L^q$ for $1 < q \le \infty$ to deduce \eqref{eq: lq1}.

\section{Step I of choosing the $\eta_Q$: Clustering}\label{sec: eta}

For each $0 \le \ell \le L$ and fixed $\vec{s}_{\leq \ell} = (s_1,\cdots,s_\ell) \in \{-1, +1\}^\ell$ we define a set $C(\vec{s}_{\leq \ell}) := \{ (z,\Delta) : z = \sum_{k=1}^L s_k \varepsilon^k, \; s_{\ell+1},\cdots, s_L \in \{-1, +1\}\} \subset E_2$. We refer to $C(\vec{s}_{\leq \ell})$ as a \emph{cluster} of $E_2$ at depth $\ell$. (By convention, in the edge case $\ell=0$, $\vec{s}_{\leq 0}$ is the empty list and $C(\vec{s}_{\leq 0}) = E_2$.) Note that the clusters of depth $L$ are singletons.

Accordingly, for fixed $\ell$ there are $2^\ell$ distinct clusters of $E_2$ of depth $\ell$. The convex hulls of distinct clusters are disjoint (assuming $\varepsilon \leq 1/2$), and thus, the clusters of depth $\ell$ inherit an order from the real number line. We shall order the clusters of $E_2$ of depth $\ell$ by $C_1^{\ell}$, $C_2^{\ell}$, \dots, $C_{2^\ell}^\ell$---the following properties are apparent:
\begin{itemize}
    \item $\#(C_i^{\ell}) = 2^{L-\ell}$,
    \item $\diam(C_i^{\ell}) \approx \varepsilon^{\ell+1}$ (for $\ell=0,\dots,L-1)$, and
    \item $\{C_i^\ell\}_{i=1}^{2^\ell}$ is a partition of $E_2$.
\end{itemize}
We write $\cC_\ell := \{ C_i^\ell\}_{i=1}^{2^{\ell}}$ to denote the set of all clusters of $E_2$ of depth $\ell$. We then define the set of all clusters
\[
\cC = \bigcup_{\ell=0}^{L} \cC_\ell.
\]

The set of all clusters forms a full, binary tree via the relation of set inclusion, with clusters of depth $\ell$ corresponding to vertices of depth $\ell$. Any $C \in \cC_\ell$ with $\ell \ge 1$ therefore has a unique \emph{parent cluster} $\pi(C)\in \cC_{\ell-1}$.

To every cluster $C \in \cC_\ell$ for $\ell=0,\dots, L-1$ we associate a ball $B_C$ and to every cluster $C \in \cC_\ell$ for $\ell=1,\dots,L-1$ we associate a ball $\hat{B}_C$. For convenience, we'll assume that $B_C$, $\hat{B}_C$ are centered at the same point and that $\text{radius}(\hat{B}_C) = 10^{M+1}\cdot \text{radius}(B_C)$ for a positive integer $M$ (independent of the cluster $C$). We require these balls to satisfy the following properties:
\begin{itemize}
\item For every $C \in \cC_\ell$ ($\ell=0,\cdots,L-1$), we have
\begin{itemize}
    \item[$\star$] $\diam(B_C) \approx \diam(C) \approx \varepsilon^{\ell+1}$, and
    \item[$\star$] $C \subset B_C$.
\end{itemize}
\item For every $C \in \cC_\ell$ ($\ell = 1, \dots, L-1)$, we have
\begin{itemize}
    \item[$\star$] $\diam(\hat{B}_C) \approx \diam(\pi(C)) \approx \varepsilon^\ell$, and
    \item[$\star$] $B_C \subset \hat{B}_C \subset B_{\pi(C)}$.
\end{itemize}
\item For distinct clusters $C, C' \in \cC_\ell$, we have
\begin{itemize}
    \item[$\star$] $\dist(B_C, B_{C'}) > c \varepsilon^{\ell}$ ($\ell=0,\dots,L-1$), and
    \item[$\star$] $\dist(\hat{B}_C, \hat{B}_{C'}) > c \varepsilon^\ell$ ($\ell=1, \dots, L-1)$.
\end{itemize}
\end{itemize}

We remark that (after taking $\varepsilon$ smaller than some absolute constant) the last bullet point implies that, for fixed $\ell$, each of the families $\{B_C\}_{C \in \cC_\ell}$ and $\{\hat{B}_C\}_{C \in \cC_\ell}$ is pairwise disjoint.

Given $f : E \rightarrow \R$, for each $x = (x^{(1)}, \Delta) \in E_2$ we define
\begin{equation}\label{eq: interp 2}
\eta_x = \frac{f(x) - f(x^{(1)},0)}{\Delta}.
\end{equation}
For each $Q \in \cz_{II}$ there is a single point in $(1.1Q \cap E_2)$, denoted by $x_Q$. We then define
\begin{equation}\label{etaQ:def1}
\eta_Q = \eta_{x_Q} \;\text{for every}\; Q \in \cz_{II}.
\end{equation}

To any square $Q \in \cz$ we associate a cluster $C_Q$ as follows. If $Q \in \cz\backslash \cz_{II}$, then we define $C_Q$ to be equal to the cluster $C \in \cC_\ell$ ($\ell=0,1,\cdots,L-1$) of smallest diameter for which $Q \subset B_{C}$ if such a cluster exists and equal to $E_2$ (recall that $E_2$ is the unique cluster of depth 0) if no such cluster exists. If $Q \in \cz_{II}$, then we set $C_Q = \{x_Q\}\in\cC_L$, where $x_Q$ is the unique point in $1.1Q \cap E_2$. Observe that for all $Q \in \partial \cz$ we have $C_Q = E_2$.

Observe that $C_Q \in \cC_L$ if and only if $1.1Q \cap E_2 \neq \emptyset$. In particular, if $1.1 Q \cap E_2 = \emptyset$ then $C_Q \in \cC_\ell$ ($\ell=0,\cdots,L-1$) has cardinality at least $2$.

For each cluster $C \in \cC$ we introduce a real number $\eta_C$.
Recall that every cluster $C \in\cC_L$ satisfies $\#(C) = 1$; we let $x_C$ denote the unique point in $C$. We then define
\begin{equation}\label{eq: a000}
\eta_C = \eta_{x_C} \;\text{for every}\; C \in \cC_L.
\end{equation}
(Recall that $\eta_{x}$ is defined by \eqref{eq: interp 2}.) We defer choosing $\{\eta_C\}_{C \in \cC_\ell}$ for $\ell=0,\dots, L-1$ until Section \ref{sec: eta tree}.

We determine the $\eta_Q$ in terms of the $\{\eta_C\}$ by specifying that
\begin{equation}\label{etaQ:def2}
\eta_Q = \eta_{C_Q}\;\text{for all} \; Q \in \cz.
\end{equation}
According to the above definitions, we have
\[
\eta_Q = \eta_{C_Q} = \eta_{x_{Q}} \mbox{ for all } Q \in \cz_{II},
\]
where $C_Q = \{x_Q\} = 1.1Q \cap E_2$. From this property, observe that, independently of how we choose the remaining $\eta_Q$, the function $F$ satisfies
\begin{equation*}
    F|_E = f,
\end{equation*}
as claimed.

We now work to bound the expression
\[
\sum_{\substack{Q,Q' \in \cz \\ Q \leftrightarrow Q'}} |\eta_Q - \eta_{Q'}|^p \delta_Q^{2-p}
\]
appearing on the right-hand side of the patching lemma (Lemma \ref{lem: patching}).

Note that any distinct $Q, Q' \in \cz$ with $Q \leftrightarrow Q'$ for which $\eta_Q \ne \eta_{Q'}$ necessarily satisfy either (1), $C_Q = \pi(C_{Q'})$, (2) $C_{Q'} = \pi(C_Q)$, or (3) $C_Q,C_{Q'} \in \cC_{L}$, $C_Q \neq C_{Q'}$ and $\pi(C_Q) = \pi(C_{Q'})$. Therefore,
\begin{equation}\label{eq: m0}
    \begin{split}
        \sum_{Q \leftrightarrow Q'} |\eta_Q - \eta_{Q'}|^p \delta_Q^{2-p} \lesssim &\sum_{ \ell=1}^L \sum_{C \in \cC_\ell} | \eta_{\pi(C)} - \eta_{C}|^p \sum_{(Q,Q') \in \cQ_C} \delta_Q^{2-p}\\
        & + \sum_{\substack{C, C' \in \cC_L:\\ \pi(C) = \pi(C')}} |\eta_{C} - \eta_{C'}|^p  \sum_{\substack{Q,Q' \in \cz \\ C_Q = C, C_{Q'} = C'}} \delta_Q^{2-p},
    \end{split}
\end{equation}
where $\cQ_C$ is defined for $C \in \cC_\ell$ ($\ell=1,\dots,L)$ by
\begin{equation*}
\cQ_C := \{(Q, Q') \in \cz \times \cz : Q \leftrightarrow Q', C_Q = C, C_{Q'} = \pi(C)\}.
\end{equation*}

We first bound the second sum on the right-hand side of \eqref{eq: m0}. Note that for $C \in \cC_L$ there are a bounded number of $Q \in \cz$ with $C_Q = C$ (any such $Q$ satisfies that $1.1Q \cap E_2 = \{x_Q\} = C$). For any such $Q$, it holds that $\delta_Q \approx \Delta = \epsilon^L$. Similarly, for $C' \in \cC_L$ there are a bounded number of $Q' \in \cz$ with $C_{Q'} = C'$. If $\pi(C) = \pi(C')$ then we can use the triangle inequality to estimate $|\eta_{C} - \eta_{C'}| \leq |\eta_C - \eta_{\pi(C)}| + |\eta_{C'} - \eta_{\pi(C')}|$. Thus,
\begin{equation}\label{term2:eqn}
\sum_{\substack{C, C' \in \cC_L:\\ \pi(C) = \pi(C')}} |\eta_{C} - \eta_{C'}|^p  \sum_{\substack{Q,Q' \in \cz \\ C_Q = C, C_{Q'} = C'}} \delta_Q^{2-p} \lesssim \epsilon^{L(2-p)} \sum_{C \in \cC_L}  | \eta_C - \eta_{\pi(C)}|^p .
\end{equation}

Note that for any cluster $C \in \cC_\ell$, $\ell = 1, \dots, L-1$, we have
\begin{equation}\label{deltaQ:prop}
\Delta/9 \leq \delta_Q \leq c \varepsilon^{\ell+1}\;\text{for any}\; Q\in \cz \;\text{for which}\; C_Q = C.
\end{equation}
Indeed, see \eqref{eq: cz 1}, and note that if $Q$ and $C$ are as above then $Q \subset B_C$ and so $\delta_Q \leq \diam(B_C) \approx \epsilon^{\ell+1}$.
%Therefore, provided the constant $\hat{C}$ above is sufficiently large, for any $Q$ with $C_Q = C$ we have
%\[
%\delta_Q \approx 10^{-k}\varepsilon^{\ell+1}
%\]
%for some integer $k$ with $0 \le k \le k^\#$; here, $k^\#$ is determined by $\varepsilon, \ell, L$ and we have $10^{-k^\#}\varepsilon^{\ell+1}\approx \Delta$. 

Observe that $\#\cQ_C \lesssim 1$ for any $C \in \cC_L$, simply because there are at most a bounded number of squares $Q$ with $C_Q = C$ (any such square satisfies $1.1Q \cap E = C_Q$), and each of these squares $Q$ touches a bounded number of squares $Q'$. Since any $Q$ with $C_Q = C $ for $C \in \cC_L$ satisfies $\delta_Q \approx \Delta = \varepsilon^L$, we get
\begin{equation}\label{eq: a00}
    \sum_{(Q,Q') \in \cQ_C}  \delta_Q^{2-p} \lesssim  \varepsilon^{L(2-p)}\;\text{for any}\; C \in \cC_L.
\end{equation}

Now fix $C \in \cC_\ell$ for $\ell=1,\dots,L-1$. We claim that
\begin{equation}\label{eq: a0}
\#\{(Q,Q')\in\cQ_C: \delta_Q = \delta\} \lesssim 1 \;\text{for every}\; \delta > 0.
\end{equation}
 Suppose $(Q,Q') \in \cQ_C$ with $ \delta_{Q}= \delta$. Then, since $Q \subset B_C$, $Q' \not\subset B_C$, and $Q\leftrightarrow Q'$, it must be the case that $Q, Q'$ are contained in a $c \delta_Q$-neighborhood of the boundary of $B_C$ for some absolute constant $c$. By \eqref{eq: cz 3}, it is also the case that $Q, Q'$ are at most distance $c' \delta_Q$ from the $x^{(1)}$-axis for another absolute constant $c'$. Here, $\delta_Q = \delta$ is fixed. A simple packing argument then yields \eqref{eq: a0}. Combining \eqref{deltaQ:prop} and \eqref{eq: a0}, we show that
\begin{equation}\label{eq: a1}
\begin{aligned}
        & \sum_{(Q,Q')\in \cQ_C} \delta_Q^{2-p} \leq \sum_{k\leq \log_2(c \epsilon^{\ell+1})} \sum_{(Q,Q')\in \cQ_C, \delta_Q = 2^k} (2^k)^{2-p} \\
        \quad &\lesssim \sum_{k\leq \log_2(c \epsilon^{\ell+1})} 2^{k(2-p)}
         \lesssim  \varepsilon^{(\ell+1)(2-p)} \quad (C \in \cC_\ell, \ell=1,\cdots,L-1).
\end{aligned}
\end{equation}
(Note that we've used the assumption $(2-p)>0$.)

Combining \eqref{term2:eqn} with \eqref{eq: a00}, \eqref{eq: a1}, allows us to continue the bound in \eqref{eq: m0},
\begin{equation}\label{eq: eo10}
\begin{split}
    \sum_{Q\touch Q'} |\eta_Q - \eta_{Q'}|^p \delta_Q^{2-p} \lesssim &\sum_{\ell=1}^{L-1} \varepsilon^{(\ell+1)(2-p)}\sum_{C \in \cC_\ell} | \eta_{\pi(C)} - \eta_C|^p\\
    & +  \varepsilon^{L(2-p)} \sum_{C \in \cC_L} | \eta_{\pi(C)} - \eta_{C}|^p.
\end{split}
\end{equation}
For $\ell=1,\dots,L$ we define
\begin{equation} \label{nu ell: defn}
\nu_\ell = \begin{cases}
    \varepsilon^{(\ell+1)(2-p)} &\text{for}\; \ell=1,\dots, L-1,\\
    \varepsilon^{L(2-p)} &\text{for}\;  \ell=L.
\end{cases}
\end{equation}
We then rewrite inequality \eqref{eq: eo10} as
\begin{equation}\label{eq: eo1}
    \sum_{Q\touch Q'} |\eta_Q - \eta_{Q'}|^p \delta_Q^{2-p} \lesssim \sum_{\ell=1}^{L} \nu_\ell \sum_{C \in \cC_\ell} | \eta_{\pi(C)} - \eta_C|^p.
\end{equation}

In the next section we will choose the remaining $\eta_C$ (for $C \in \cC_\ell$ with $\ell=0,1,\dots,L-1$) and prove that our choices satisfy suitable estimates. Finally, in the last section we will explain how to use those estimates to complete the proof of Theorem \ref{thm: main}.

\section{Step II of choosing the $\eta_Q$: Minimizing $\|F\|_{L^{2,p}}$}\label{sec: eta tree}

Recall that we have already defined $\eta_C$ for $C \in \cC_{L}$ in \eqref{eq: a000}.

The goal of this section is to choose $\eta_C$ for each $C \in \cC_{\ell}$ ($\ell=0,1,\dots,L-1)$ so that for any $G \in L^{2,p}(\R^2)$ with $G|_E =f$ we have
\begin{equation}\label{eq: et1}
    \sum_{\ell=1}^{L} \nu_\ell \sum_{C \in \cC_\ell} | \eta_{\pi(C)} - \eta_C|^p \lesssim \|G\|_{L^{2,p}(\R^2)}^p.
\end{equation}
This is easy to do, thanks to Theorem \ref{thm: fk}. Recall $\nu_\ell$ are defined in \eqref{nu ell: defn}.

The following is an immediate corollary of Theorem \ref{thm: fk}. 

\begin{cor}\label{cor: fk}
    There exist $\eta_C$ for each $C \in \cC_\ell$ ($\ell = 1, \dots, L-1)$, depending linearly on $\{\eta_C\}_{C \in \cC_{L}},$ so that the following holds. Let $\{\gamma_C\}_{C\in \cC}$ be a collection of real numbers with $\gamma_C = \eta_C$ for each $C \in \cC_{L}$. Then
    \[
    \sum_{\ell=1}^{L} \nu_\ell \sum_{C \in \cC_\ell} | \eta_{\pi(C)} - \eta_C|^p \lesssim \sum_{\ell=1}^{L} \nu_\ell \sum_{C \in \cC_\ell} | \gamma_{\pi(C)} - \gamma_C|^p.
    \]
\end{cor}

We fix the definition of $\{\eta_C\}_{C \in \cC}$ as in Corollary \ref{cor: fk}. Recall that $\{\eta_C\}_{C \in \cC_{L}}$ depend linearly on the data $f$ -- see \eqref{eq: a000} and \eqref{eq: interp 2}. Thus, $\eta_C$ depends linearly on $f$ for each $C \in \cC$. 

We  claim that for any $G \in L^{2,p}(\R^2)$ with $G|_E = f$ we have
\begin{equation}\label{eq: et2}
    \sum_{\ell=1}^{L-1} \varepsilon^{(\ell+1)(2-p)}\sum_{C \in \cC_\ell} | (\partial_2 G)_{B_{\pi(C)}} - (\partial_2 G)_{B_C}|^p\lesssim \|G\|_{L^{2,p}(\R^2)}^p.
\end{equation}
We will later show that \eqref{eq: et2} and Corollary \ref{cor: fk} allow us to deduce estimate \eqref{eq: et1} for the $\{\eta_C\}_{C \in \cC}$, completing the task of this section.

For every $C \in \cC_\ell$, $\ell=1,\dots, L-1$, let $r_C, x_C$ denote the radius and center, respectively, of the ball $B_C$. For $k \ge 0$ define the annulus
\[
A_{C,k} = \{ x \in \R^2 : 10^k r_C \le |x - x_C| < 10^{k+1} r_C\}.
\]
Recall (see Section \ref{sec: eta}) that $\hat{B}_C$ is also centered at $x_C$ and that
\[
\text{radius}(\hat{B}_C) = 10^{M+1} r_C.
\]
for some positive integer $M$. Define another annulus
\[
A_C = \bigcup_{k = 0 }^{M} A_{C,k}.
\]
Note that the collection of annuli $\{A_C\}_{C \in \cC_\ell, \ell=1,\dots L-1}$ is pairwise disjoint.

We claim that for any $1<r<p$ and for every $C\in \cC_\ell$ ($\ell = 1,\dots, L-1$) we have
\begin{equation}\label{eq: et3}
    |(\partial_2 G)_{B_{\pi(C)}} - (\partial_2 G)_{B_C}|^p \varepsilon^{(\ell+1)(2-p)} \lesssim \| \cM( | \nabla^2 G|^r) \|_{L^{p/r}(A_C)}^{p/r}.
\end{equation}
Since the $A_C$ are pairwise disjoint, \eqref{eq: et3} implies
\[
\sum_{\ell=1}^{L-1} \varepsilon^{(\ell+1)(2-p)}\sum_{C \in \cC_\ell} | (\partial_2 G)_{B_{\pi(C)}} - (\partial_2 G)_{B_C}|^p\lesssim \| \cM(|\nabla^2 G|^r) \|_{L^{p/r}(\R^2)}^{p/r};
\]
we use the boundedness of the Hardy-Littlewood maximal operator from $L^{p/r}$ to $L^{p/r}$ to deduce \eqref{eq: et2}.

To prove \eqref{eq: et3}, use Lemma \ref{lem: poincare} to get
\begin{equation*}
    \begin{split}
        |(\partial_2 G)_{B_C} - (\partial_2G)_{A_{C,0}}| &\lesssim \varepsilon^{(\ell+1)(1-2/r)}\|G\|_{L^{2,r}(B_C \cup A_{C,0})}\\
        &\lesssim \varepsilon^{\ell+1}(\cM(|\nabla^2 G|^r(z))^{1/r}
    \end{split}
\end{equation*}
for any $z \in A_{C,0}$. (Recall that $r_C \approx \varepsilon^{\ell+1}$.) Taking $p$-th powers and integrating over $A_{C,0}$, we get
\[
|(\partial_2 G)_{B_C} - (\partial_2G)_{A_{C,0}}|^p \varepsilon^{(\ell+1)(2-p)}\lesssim \| \cM(|\nabla^2 G|^r)\|_{L^{p/r}(A_{C,0})}^{p/r}.
\]
Similarly, we can show
\begin{align*}
    & |(\partial_2 G)_{A_{C,k}} - (\partial_2 G)_{A_{C,k+1}} |^p \varepsilon^{(\ell+1)(2-p)} \lesssim 10^{k(p-2)} \| \cM( |\nabla^2 G|^r) \|_{L^{p/r}(A_{C,k})}^{p/r},\\
    & |(\partial_2 G)_{B_{\pi(C)}} - (\partial_2 G)_{A_{C,M}} |^p \varepsilon^{(\ell+1)(2-p)} \lesssim 10^{M(p-2)} \|\cM(|\nabla^2 G|^r) \|_{L^{p/r}(A_{C,M})}^{p/r}.
\end{align*}
We apply the triangle inequality and use that $p<2$ to get
\begin{equation*}
    \begin{split}
     |(\partial_2 G)_{B_C}-(\partial_2 G)_{B_{\pi(C)}}| \varepsilon^{(\ell+1)(2/p-1)} &\lesssim \|\cM(|\nabla^2 G|^r)\|^{1/r}_{L^{p/r}(A_C)} \sum_{k=0}^M 10^{k(1-2/p)}\\
     & \lesssim \|\cM(|\nabla^2 G|^r)\|^{1/r}_{L^{p/r}(A_C)},
    \end{split}
\end{equation*}
proving \eqref{eq: et3}, and thus, proving \eqref{eq: et2}.

We now explain how to prove \eqref{eq: et1}. We fix $G \in L^{2,p}(\R^2)$ with $G|_E = f$ and define $(\gamma_C)_{C \in \cC}$ in terms of $G$ and $(\eta_C)_{C \in \cC}$. We take $\gamma_C = (\partial_2 G)_{B_C}$ if $C \in \cC_\ell$ for $\ell = 0,\dots,L-1$ and $\gamma_C = \eta_C$ if $C \in \cC_{L}$. Now, we apply Corollary 1 to deduce that
\begin{equation}\label{eq: b0}
\begin{split}
\sum_{\ell=1}^{L} \nu_\ell \sum_{C \in \cC_{\ell}} | \eta_{\pi(C)} - \eta_C|^p & \lesssim \sum_{\ell=1}^{L-1} \varepsilon^{(\ell+1)(2-p)} \sum_{C \in \cC_{\ell}} | (\partial_2 G)_{B_{\pi(C)}} - (\partial_2 G)_{B_C}|^p \\
& \qquad + \varepsilon^{L(2-p)} \sum_{C \in \cC_{L}} | (\partial_2 G)_{B_{\pi(C)}} - \eta_C|^p.
\end{split}
\end{equation}

Let $C \in \cC_L$. Observe that $\eta_C = \eta_{x_C} = \frac{f(x_C) - f(x_C^{(1)},0)}{\Delta} = \frac{G(x_C) - G(x_C^{(1)},0)}{\Delta}$. By Lemma \ref{lem: poincare}, we easily get that
\[
|(\partial_2 G)_{B_{\pi(C)}} - \eta_C|^p \varepsilon^{L(2-p)} \lesssim ||G||_{L^{2,p}(B_{\pi(C)})}^p\;\text{for any}\; C \in \cC_L.
\]
The collection of balls $\{B_{\pi(C)}\}_{C \in \cC_L}$ has bounded overlap, so
\[
\varepsilon^{L(2-p)} \sum_{C \in \cC_{L}} | (\partial_2 G)_{B_{\pi(C)}} - \eta_C|^p \lesssim ||G||_{L^{2,p}(\R^2)}^p.
\]
Combining this estimate with \eqref{eq: b0} and \eqref{eq: et2}, we complete the proof of \eqref{eq: et1}.

\section{Proof of Main Theorem}

Combining Lemma \ref{lem: patching} and the inequalities \eqref{eq: lq1}, \eqref{eq: eo1}, and \eqref{eq: et1}  proves the following lemma.

\begin{lem}\label{lem: F0}
    Let $\{\eta_C\}_{C \in \cC}$ be defined via \eqref{eq: a000} and Corollary \ref{cor: fk}, and let $\{\eta_Q\}_{Q \in \cz}$ be defined in terms of $\{\eta_C\}_{C \in \cC}$ in \eqref{etaQ:def2}. Then the function $F\in L^{2,p}(Q^0)$ defined by \eqref{eq: interp 1} satisfies:
    \begin{itemize}
        \item $F |_E = f$
        \item For any $G \in L^{2,p}(\R^2)$ with $G|_E = f$, we have
        \[
        \|F\|_{L^{2,p}(Q^0)} \lesssim \|G\|_{L^{2,p}(\R^2)}.
        \]
    \end{itemize}
\end{lem}

It is easy to deduce Theorem \ref{thm: main} from Lemma \ref{lem: F0}. Recall that there is some affine polynomial $L_0$ such that
\[
L_Q = L_0\;\text{for every}\; Q \in \partial \cz.
\]
Similarly, recall that for every $Q\in \partial \cz$ we have $C_Q = E_2$. Writing $\eta_0 = \eta_{E_2}$, we thus have 
\[
\eta_Q = \eta_{0}\;\text{for every}\;Q \in \partial \cz.
\]

We define a function $\tilde{F}: \R^2 \rightarrow \R$ by setting
\[
\tilde{F}(x) = \begin{cases}
    F(x) &\text{if } x \in Q^0,\\
    L_{0}(x) + \eta_{0}\cdot x^{(2)} & \text{if } x \notin Q^0.
\end{cases}
\]
Clearly, $\tilde{F}|_E = f$ and $\|\tilde{F}\|_{L^{2,p}(\R^2)} = \| F\|_{L^{2,p}(Q^0)}$. This proves Theorem \ref{thm: main}.

\bibliographystyle{plain}
\bibliography{ref}

\begin{thebibliography}{1}

\bibitem{fefferman2020fitting}
Charles Fefferman and Arie Israel.
\newblock {\em Fitting Smooth Functions to Data}, volume 135 of {\em CBMS Regional Conference Series in Mathematics}.
\newblock AMS, 2020.

\bibitem{fefferman2014sobolev}
Charles Fefferman, Arie Israel, and Garving Luli.
\newblock Sobolev extension by linear operators.
\newblock {\em J. Am. Math. Soc.}, 27(1):69--145, 2014.

\bibitem{fefferman2023linear}
Charles Fefferman and Bo’az Klartag.
\newblock Linear extension operators for sobolev spaces on radially symmetric binary trees.
\newblock {\em Adv. Nonlinear Stud.}, 23(1):20220075, 2023.

\bibitem{gilbarg1977elliptic}
David Gilbarg and Neil~S Trudinger.
\newblock {\em Elliptic Partial Differential Equations of Second Order}.
\newblock Springer-Verlag Berlin, 1 edition, 1977.

\bibitem{stein1993harmonic}
Elias~M Stein.
\newblock {\em Harmonic Analysis: Real-Variable Methods, Orthogonality, and Oscillatory Integrals}, volume~43 of {\em Princeton Mathematical Series}.
\newblock Princeton University Press, 1993.

\end{thebibliography}

\end{document}